\documentclass[11 pt]{amsart}
\usepackage{latexsym,amscd,amssymb, graphicx, amsthm}  
\usepackage{blkarray}
\usepackage{tikz-cd}

\usepackage[margin=1in]{geometry}

\numberwithin{equation}{section}

\newtheorem{theorem}{Theorem}[section]
\newtheorem{proposition}[theorem]{Proposition}
\newtheorem{corollary}[theorem]{Corollary}
\newtheorem{lemma}[theorem]{Lemma}

\newtheorem{question}[theorem]{Question}
\newtheorem{problem}[theorem]{Problem}

\newtheorem{remark}[theorem]{Remark}

\newtheorem{defn}[theorem]{Definition}
\theoremstyle{definition}

\newcommand{\GL}{{\mathrm {GL}}}

\newcommand{\Hilb}{{\mathrm {Hilb}}}
\newcommand{\grFrob}{{\mathrm {grFrob}}}

\newcommand{\Frob}{{\mathrm {Frob}}}

\newcommand{\symm}{{\mathfrak{S}}}

\newcommand{\wt}{{\mathrm{wt}}}

\newcommand{\CC}{{\mathbb {C}}}

\newcommand{\ZZ}{{\mathbb {Z}}}
\newcommand{\PP}{{\mathbb {P}}}

\newcommand{\vv}{{\mathbf {v}}}

\newcommand{\LM}{{\textsc {lm}}}

\begin{document}

\title[Lefschetz theory for exterior algebras and fermionic diagonal coinvariants]
{Lefschetz theory for exterior algebras and fermionic diagonal coinvariants}

\author{Jongwon Kim}
\address
{Department of Mathematics \newline \indent
University of Pennsylvania \newline \indent
Philadelphia, PA, 19104-6395, USA}
\email{jk1093@math.upenn.edu}

\author{Brendon Rhoades}
\address
{Department of Mathematics \newline \indent
University of California, San Diego \newline \indent
La Jolla, CA, 92093-0112, USA}
\email{bprhoades@math.ucsd.edu}

\begin{abstract}
Let $W$ be an irreducible complex reflection group acting on its reflection representation $V$.
We consider the doubly graded action of $W$ on the exterior algebra $\wedge (V \oplus V^*)$
as well as its quotient $DR_W := \wedge (V \oplus V^*)/ \langle \wedge (V \oplus V^*)^{W}_+ \rangle$
by the ideal generated by its homogeneous $W$-invariants with vanishing constant term.
We describe the bigraded isomorphism type of $DR_W$; when $W = \symm_n$ is the symmetric group, the
answer is a difference of Kronecker products of hook-shaped $\symm_n$-modules.
We relate the Hilbert series of $DR_W$ to the (type A) Catalan and Narayana numbers and describe 
a standard monomial basis of $DR_W$ using a variant of Motzkin paths.
Our methods are type-uniform and involve a Lefschetz-like theory which applies to the 
exterior algebra $\wedge (V \oplus V^*)$.
\end{abstract}

\maketitle

\section{Introduction}
\label{Introduction}

Let $\CC[x_1, \dots, x_n, y_1, \dots, y_n]$ be a polynomial ring in $2n$ variables equipped with 
the diagonal action of the symmetric group $\symm_n$:
\begin{equation}
w.x_i := x_{w(i)} \quad \text{and} \quad w.y_i := y_{w(i)} 
\end{equation}
for all $w \in \symm_n$ and $1 \leq i \leq n$.
The quotient of $\CC[x_1, \dots, x_n, y_1, \dots, y_n]$ by the ideal generated by the homogeneous
$\symm_n$-invariants of positive degree 
 is the {\em diagonal coinvariant ring}; its bigraded $\symm_n$-structure was calculated by 
Haiman \cite{Haiman} using algebraic geometry.

In the last couple years, algebraic combinatorialists have studied variations
of the diagonal coinvariants 
involving sets of  commuting and anti-commuting variables \cite{Bergeron, BRT, HRS, PR, RW, SW, Zabrocki}.
In this paper we completely describe the bigraded $\symm_n$-structure of the diagonal 
coinvariants involving two sets of anti-commuting variables (but no commuting variables). 
Our methods apply equally well (and uniformly) to any irreducible complex reflection group $W$\footnote{And, 
in fact, to a wider class of groups $G$; see Remark~\ref{other-g}.}
 as to
the symmetric group $\symm_n$.

Let $W$ be an irreducible complex reflection group of rank $n$ acting on its reflection representation
$V \cong \CC^n$.
The action of $W$ on $V$ induces an action of $W$ on 
\begin{itemize}
\item the dual space $V^* = \mathrm{Hom}_{\CC}(V, \CC)$,
\item the direct sum $V \oplus V^*$ of $V$ with its dual space, and finally
\item the exterior algebra $\wedge (V \oplus V^*)$ over the  $2n$-dimensional vector space $V \oplus V^*$.
\end{itemize}
By placing $V$ in bidegree $(1,0)$ and $V^*$ in bidegree $(0,1)$, this last space 
$\wedge (V \oplus V^*)$ attains the structure of a doubly graded $W$-module.

If we let $\Theta_n = (\theta_1, \dots, \theta_n)$ be a basis for $V$ and 
$\Xi_n = (\xi_1, \dots, \xi_n)$ be a basis for $V^*$,
we have a natural identification
\begin{equation}
\wedge (V \oplus V^*) = \wedge \{ \Theta_n, \Xi_n \}
\end{equation}
of $\wedge (V \oplus V^*)$ with the exterior algebra $\wedge \{ \Theta_n, \Xi_n \}$ generated by the 
symbols $\theta_i$ and $\xi_i$ over $\CC$.
Following the terminology of physics, we refer to the $\theta_i$ and $\xi_i$ as {\em fermionic} variables.
In physics, such variables are used to model fermions, with relations $\theta_i^2 = \xi_i^2 = 0$
corresponding to the Pauli Exclusion Principle: no two fermions can occupy the same state at the same time\footnote{A 
commuting variable $x_i$ is called {\em bosonic}; the power $x_i^2$ corresponds to
two indistinguishable bosons in State $i$.}.
The model $\wedge \{ \Theta_n, \Xi_n \}$ for $\wedge (V \otimes V^*)$ will be helpful in our arguments.
The following quotient ring is our object of study.

\begin{defn}
\label{fermionic-coinvariants}
The {\em fermionic diagonal coinvariant ring} is the quotient
\begin{equation}
DR_W :=  \wedge (V \oplus V^*) / \langle \wedge (V \oplus V^*)^{W}_+ \rangle
\end{equation}
of $\wedge (V \oplus V^*)$ by the (two-sided) ideal generated by the subspace 
$\wedge (V \oplus V^*)^{W}_+ \subseteq \wedge (V \oplus V^*) $ of $W$-invariant elements with 
vanishing constant term.
\end{defn}

The ideal $\langle \wedge (V \oplus V^*)^{W}_+ \rangle$ is $W$-stable and bihomogeneous, so
the quotient ring $DR_W$ has the structure of a bigraded $W$-module.
We will see (Proposition~\ref{casimir-generation}) that this ideal is principal, generated by 
a `Casimir element' $\delta_W \in V \oplus V^*$.
Our results are as follows.

\begin{itemize}
\item  We describe the bigraded $W$-isomorphism type of $DR_W$ in terms 
of the isomorphism types of the exterior powers $\wedge^i V$ and $\wedge^j V^*$
(Theorem~\ref{fermion-module-structure}).
\item  We show that $\dim DR_W = {2n+1 \choose n}$ whenever $W$ has rank $n$ and relate
the dimensions of its graded pieces to Catalan and Narayana numbers
(Corollaries~\ref{dimension-of-dr} and \ref{dimension-of-dr-catalan}).
\item  We describe an explicit monomial basis of $DR_W$ using  a variant of Motzkin paths
(Theorem~\ref{standard-monomial-basis}) and describe the bigraded Hilbert series of 
$DR_W$ in terms of the combinatorics of these paths 
(Corollary~\ref{hilbert-motzkin}).
\item When $W = \symm_n$, in Section~\ref{Permutation} we give variants of the above results 
as they apply to the $n$-dimensional permutation representation of $\symm_n$
(as opposed to its $(n-1)$-dimensional reflection representation).
\end{itemize}

The key tool in our analysis is the realization (Theorem~\ref{delta-lefschetz}) of the Casimir generator 
$\delta_W$ of the ideal defining $DR_W$ as a kind of `$W$-invariant Lefschetz element'
in the ring $\wedge (V \oplus V^*)$.   The ring 
$\wedge (V \oplus V^*)$, similar to the cohomology ring of a compact smooth complex manifold,
satisfies `bigraded' versions of Poincar\'e Duality and the Hard Lefschetz Theorem.
This is somewhat unusual on two counts.
\begin{itemize}
\item Any homogeneous linear form in an exterior algebra squares to zero, and hence is not well-suited
to be a (strong) Lefschetz element.
\item Lefschetz elements arising in 
coinvariant theory are rarely $W$-invariant. For example, if $W$ is a Weyl group with associated
complete flag manifold $G/B$ we may present the cohomology of $G/B$ as 
\begin{equation}
H^{\bullet}(G/B; \CC) = \CC[\mathfrak{h}] / \langle \CC[\mathfrak{h}]^W_+ \rangle
\end{equation}
where $\mathfrak{h}$ is the Cartan subalgebra of the Lie algebra $\mathfrak{g}$ of $G$.
An element $\ell \in  \CC[\mathfrak{h}]_1 = \mathfrak{h}^*$ is a Lefschetz element if and only if it is not fixed by any element of $W$
\cite{MNW}.  So the Lefschetz property is in some sense opposite to $W$-invariance in this case.
\end{itemize}
For examples of coinvariant-like quotients
 of {\em superspace} $\CC[V] \otimes \wedge V^* $ satisfying other nontraditional
bigraded
versions of Poincar\'e Duality and (conjecturally) Hard Lefschetz, see \cite{RW}.

The remainder of the paper is organized as follows.
In {\bf Section~\ref{Background}} we give background material on complex reflection
groups, Gr\"obner theory associated to exterior algebras, and 
the representation theory of $\symm_n$.
In {\bf  Section~\ref{Lefschetz}} we prove that $\wedge (V \oplus V^*)$
satisfies bigraded versions of the Hard Lefschetz Property and Poincar\'e Duality.
This builds on work of Hara and Watanabe \cite{HW} showing that the incidence matrix between
complementary ranks of the Boolean poset $B(n)$ is invertible.
In {\bf Section~\ref{Casimir}} we apply
these Lefschetz results to determine the bigraded $W$-structure of $DR_W$.
In {\bf Section~\ref{Motzkin}} we describe the standard monomial basis of $DR_W$ using 
lattice paths.
In {\bf Section~\ref{Permutation}} we specialize to $W = \symm_n$ and translate our results 
to the setting of the permutation representation of $\symm_n$.
We close in {\bf Section~\ref{Open}} with some open problems.

\section{Background}
\label{Background}

\subsection{Complex reflection groups}
Let $V = \CC^n$ be an $n$-dimensional complex vector space. An element 
$t \in \GL(V) = \GL_n(\CC)$ is a {\em reflection} if its fixed space 
$V^t := \{ v \in V \,:\, t(v) = v \}$ satisfies $\dim V^t = n-1$.

A finite subgroup $W \subseteq GL(V)$ is a {\em complex reflection group} if 
it is generated by reflections.
The $W$-module $V$ is called the {\em reflection representation} of $W$.
The dimension $\dim V = n$ of $V$ is called the {\em rank} of $W$.

If $W_1$ and $W_2$ are reflection groups with reflection representations $V_1$ and $V_2$,
the direct product $W_1 \times W_2$ is naturally a reflection group with reflection representation $V_1 \oplus V_2$.
A reflection group $W$ acting on $V$ is  {\em irreducible} if it is impossible to express $W$
as a direct product $W_1 \times W_2$ of reflection groups acting on $V = V_1 \oplus V_2$ unless $V_1 = 0$
or $V_2 = 0$.

\subsection{Exterior Gr\"obner theory}
Let $\Theta_n = (\theta_1, \dots, \theta_n)$ be a list on $n$ anticommuting variables and let 
$\wedge \{ \Theta_n \}$ be the exterior algebra generated by these variables over $\CC$.
For any subset $S = \{i_1 < i_2 <  \cdots < i_k \} \subseteq \{1, 2, \dots , n \}$, we let 
\begin{equation}
\theta_S := \theta_{i_1} \theta_{i_2} \cdots \theta_{i_k}
\end{equation}
where the multiplication is in increasing order of subscripts.
We refer to the $\theta_S$ as {\em monomials};
the set $\{ \theta_S \,:\, S \subseteq \{1, 2, \dots, n \} \}$ is the monomial basis
of $\wedge \{ \Theta_n \}$.  
Given two monomials $\theta_S$ and $\theta_T$, we write $\theta_S \mid \theta_T$ to mean
$S \subseteq T$.

A total order $<$ on the set $\{ \theta_S \,:\, S \subseteq \{1, 2, \dots, n \} \}$ is a {\em term order} if 
\begin{itemize}
\item we have 
$1 = \theta_{\varnothing} \leq \theta_S$ for all $S$ and
\item for all subsets
 $S, T, U$  with $U \cap S = U \cap T = \varnothing$, 
$\theta_S < \theta_T$ implies $\theta_{S \cup U} < \theta_{T \cup U}$.
\end{itemize}

Given a term order $<$,
for any nonzero element $f \in \wedge \{ \Theta_n \}$, let $\LM(f)$ be the largest monomial 
$\theta_S$ under the total order $<$ such that $\theta_S$ appears with nonzero coefficient in $f$.
If $I \subseteq \wedge \{ \Theta_n \}$ is a two-sided ideal, let 
\begin{equation*}
\LM(I) := \{ \LM(f) \,:\, f \in I - \{0\} \}
\end{equation*}
stand for the set of leading monomials of nonzero elements in $I$.
The collection of {\em standard monomials} (or {\em normal forms}) for $I$ is 
\begin{equation}
N(I) := \{ \text{monomials $\theta_S \,:\, S \subseteq \{1, 2, \dots, n \}  $ and $  \theta_S \notin \LM(I)$} \}.
\end{equation}
The set $N(I)$ of monomials descends to a $\CC$-basis of the quotient 
$\wedge \{ \Theta_n \}/I$; this is the {\em standard monomial basis} with respect to $<$
(see for example \cite{Bremner}).

\subsection{Representation Theory}
If $V = \bigoplus_{i,j \geq 0} V_{i,j}$ is a bigraded vector space with each piece $V_{i,j}$ finite-dimensional,
the {\em bigraded Hilbert series} is 
$\Hilb(V;q,t) := \sum_{i,j \geq 0} \dim V_{i,j} \cdot q^i t^j$.  This is a formal power series in $q$ and $t$.

The irreducible representations of the symmetric group $\symm_n$ 
are in one-to-one correspondence with partitions $\lambda \vdash n$.
Given $\lambda \vdash n$, let $S^{\lambda}$ be the corresponding $\symm_n$-irreducible.
For example, the trivial representation is $S^{(n)}$ and the sign representation is $S^{(1^n)}$.

Let $\Lambda$ denote the ring of symmetric functions and let $\{ s_{\lambda} \,:\, \text{$\lambda$ a partition} \}$
denote its Schur basis.  The {\em Hall inner product} on $\Lambda$ declares the Schur basis to be orthonormal:
\begin{equation}
\langle s_{\lambda}, s_{\mu} \rangle = \delta_{\lambda,\mu}
\end{equation}
for any partitions $\lambda$ and $\mu$.

Any finite-dimensional $\symm_n$-module $U$  may be expressed uniquely as a direct sum
$U \cong \bigoplus_{\lambda \vdash n} c_{\lambda} S^{\lambda}$ for some multiplicities $c_{\lambda}$.
The {\em Frobenius image} of $U$ is the symmetric function
\begin{equation}
\Frob(U) := \sum_{\lambda \vdash n} c_{\lambda} s_{\lambda},
\end{equation}
where $s_{\lambda}$ is the Schur function.  The {\em Kronecker product} of two Schur functions 
$s_{\lambda}$ and $s_{\mu}$ for $\lambda, \mu \vdash n$ is defined by
\begin{equation}
s_{\lambda} * s_{\mu} = \Frob(S^{\lambda} \otimes S^{\mu})
\end{equation}
where $\symm_n$ acts diagonally on $S^{\lambda} \otimes S^{\mu}$.

If $V = \bigoplus_{i,j \geq 0} V_{i,j}$ is a bigraded $\symm_n$-module with each piece 
$V_{i,j}$ finite-dimensional, the {\em bigraded Frobenius image} is 
\begin{equation}
\grFrob(V;q,t) = \sum_{i,j \geq 0} \Frob(V_{i,j}) \cdot q^i t^j.
\end{equation}
This is a formal power series in $q$ and $t$ with coefficients in the ring of symmetric functions.

\section{Lefschetz theory for exterior algebras}
\label{Lefschetz}

Let $A = \bigoplus_{i = 0}^n A_i$ be a commutative graded $\CC$-algebra.  The algebra $A$ satisfies 
 {\em Poincar\'e Duality} (PD) if $A_n \cong \CC$ and if the multiplication map
$A_i \otimes A_{n-i} \rightarrow A_n \cong \CC$ is a perfect pairing for $0 \leq i \leq n$.  
In particular, this implies that the Hilbert series of $A$ is palindromic, i.e.
$\dim A_i = \dim A_{n-i}$.

If $A = \bigoplus_{i = 0}^n A_i$ satisifes PD, an element $\ell \in A_1$ is called a {\em (strong) Lefschetz element}
if for every $0 \leq i \leq n/2$, the linear map
\begin{equation}
\ell^{n-2i} \cdot (-): A_i \rightarrow A_{n-i}
\end{equation}
given by multiplication by $\ell^{n-2i}$ is bijective.  If $A$ has a Lefschetz element,
it is said to satisfy the {\em Hard Lefschetz Property} (HLP).

Algebras $A$ which satisfy PD and the HLP arise naturally in geometry as the cohomology rings
(with adjusted grading) of smooth complex projective varieties. 
The HLP of the cohomology ring of the $n$-fold product $\PP^1 \times \cdots \times \PP^1$ 
of $1$-dimensional complex projective space with itself was studied combinatorially by 
Hara and Watanabe \cite{HW}; we state their results below.

Recall that the {\em Boolean poset} $B(n)$ is the partial order on all subsets $S \subseteq \{1, \dots, n \}$ given by
$S \leq T$ if and only if $S \subseteq T$.  The poset $B(n)$ is graded, with the $i^{th}$ rank given by the family
$B(n)_i$ of $i$-element subsets of $\{1, \dots, n \}$.
Hara and Watanabe proved that the incidence matrix between complementary ranks of $B(n)$ is invertible.

\begin{theorem} 
\label{boolean-incidence}  (Hara-Watanabe \cite{HW})
Given $r \leq s \leq n$, define a ${n \choose s} \times {n \choose r}$ matrix $M_n(r,s)$ with rows
indexed by $B(n)_s$ and columns indexed by $B(n)_r$ with entires
\begin{equation}
M_n(r,s)_{T,S} = \begin{cases}
1 & \text{if $S \subseteq T$} \\
0 & \text{otherwise.}
\end{cases}
\end{equation}
For any $0 \leq i \leq n$, the square matrix $M_n(i,n-i)$ is invertible.
\end{theorem}

For example, if $n = 4$ and $i = 1$, Theorem~\ref{boolean-incidence} asserts that the $0,1$-matrix
$M_4(1,3)$ given by
\[
\begin{blockarray}{ccccc}
& $\{1\}$ & $\{2\}$ & $\{3\}$ & $\{4\}$  \\
\begin{block}{c(cccc)}
$\{1,2,3\}$ & 1 & 1 & 1 & 0  \\
$\{1,2,4\}$ & 1 & 1 & 0 & 1  \\
$\{1,3,4\}$ &   1 & 0 & 1 & 1  \\
$\{2,3,4\}$ &  0 & 1 & 1 & 1 \\
\end{block}
\end{blockarray}
 \]
 is invertible.
 
The cohomology ring of the $n$-fold product $\PP^1 \times \cdots \times \PP^1$  may be presented as 
\begin{equation}
H^{\bullet}(\PP^1 \times \cdots \times \PP^1; \CC) = \CC[x_1, \dots, x_n]/\langle x_1^2, \dots, x_n^2 \rangle
\end{equation}
where $x_i$ represents the Chern class $c_1(\mathcal{L}_i)$ and $\mathcal{L}_i$ is the dual 
of the tautological line bundle over the $i^{th}$ factor of $\PP^1 \times \cdots \times \PP^1$.
Hara and Watanabe used Theorem~\ref{boolean-incidence} to give a combinatorial proof of the fact that 
$x_1 + \cdots + x_n$ is a Lefschetz element for the ring 
$H^{\bullet}(\PP^1 \times \cdots \times \PP^1; \CC)$ \cite{HW}.

We want to study PD and the HLP in the context of the exterior algebra
$\wedge \{ \Theta_n, \Xi_n \}$.  This algebra satisfies a natural bigraded version of Poincar\'e Duality:
the top bidegree $\wedge \{ \Theta_n, \Xi_n \}_{n,n}$ is $1$-dimensional and the multiplication map
\begin{equation}
\wedge \{ \Theta_n, \Xi_n \}_{i,j} \otimes \wedge \{ \Theta_n, \Xi_n \}_{n-i,n-j} \rightarrow 
\wedge \{ \Theta_n, \Xi_n \}_{n,n} \cong \CC
\end{equation}
is a perfect pairing for any $0 \leq i, j \leq n$.

The notion of a Lefschetz element in $\wedge \{ \Theta_n, \Xi_n \}$ is a bit more subtle because
any linear form $\ell$ in the variables $\theta_1, \dots, \theta_n, \xi_1, \dots, \xi_n$ 
satisfies $\ell^2 = 0$. 
To get around this, we introduce the element
\begin{equation}
\delta_n := \theta_1 \xi_1 + \theta_2 \xi_2 +  \cdots + \theta_n \xi_n \in \wedge \{ \Theta_n, \Xi_n \}_{1,1}
\end{equation}
The following result states that $\delta_n$ is a bigraded version of a Lefschetz element
for the ring $\wedge \{ \Theta_n, \Xi_n \}$.

\begin{theorem}
\label{delta-lefschetz}
Suppose $i + j \leq n$ and let $r = n-i-j$. The linear map 
\begin{equation}
\varphi: \wedge \{\Theta_n, \Xi_n \}_{i,j}
 \xrightarrow{\, \, \, \delta_n^r \cdot \, \, \, } 
 \wedge \{\Theta_n, \Xi_n \}_{n-j,n-i}
\end{equation}
given by multiplication by $\delta_n^r$
is bijective.
\end{theorem}

\begin{proof}
The idea is to introduce strategically chosen bases of the domain and target of $\varphi$ and show 
that the matrix representing $\varphi$ with respect to these bases is invertible using
Theorem~\ref{boolean-incidence}.

Given two subsets
$A, B \subseteq \{1, \dots, n\}$, write
\begin{align*}
A - B &= \{ a_1 < a_2 < \cdots < a_r \} \\
B - A &= \{ b_1 < b_2 < \cdots < b_s\} \\
A \cap B &= \{ c_1 < c_2 < \cdots < c_t \}
\end{align*}
and set
\begin{equation}
\vv(A,B) := \xi_{c_1} \theta_{c_1} \xi_{c_2} \theta_{c_2} \cdots \xi_{c_t} \theta_{c_t}  \cdot
\theta_{a_1} \theta_{a_2} \cdots \theta_{a_r} \cdot \xi_{b_1} \xi_{b_2} \cdots \xi_{b_s}.
\end{equation}
The family
$\{ \vv(A,B) \,:\, A, B \subseteq \{1, \dots, n\} \}$ is a basis of $\wedge \{ \Theta_n, \Xi_n \}$. 
For any sets $A$ and $B$, a direct computation shows
\begin{equation}
\label{basic-transition}
\delta_n \cdot \vv(A,B) = \sum_{c \notin A \cup B} \vv(A \cup c, B \cup c).
\end{equation}
The somewhat unusual variable order in the product $\vv(A,B)$ was chosen strategically so that 
Equation~\eqref{basic-transition} does not contain any signs.

Now suppose $|A| = i$ and $|B| = j$ for $i + j \leq n$ and set $r = n - i - j$.
Iterating Equation~\eqref{basic-transition} yields
\begin{equation}
\label{transition-equation}
\delta_n^r \cdot \vv(A,B) = 
\sum_{\substack{|C| = n-j, \, |D| = n-i \\ A \subseteq C, \, \, B \subseteq D \\ |C \cap D| - |A \cap B| = r}}  \vv(C,D).
\end{equation}

We need to show that the (square) matrix of dimensions 
${n \choose i} \cdot {n \choose j} \times {n \choose n-j} \cdot {n \choose n-i}$ defined by the system
 \eqref{transition-equation} is invertible.
 For any $\vv(C,D)$ appearing in on the right-hand side of \eqref{transition-equation}
we have $ A - B = C - D =: I$ and $B - A = D - C =: J$. 
The matrix representing $\varphi$ therefore 
breaks up as a direct sum of smaller matrices indexed by the two sets
$I$ and $J$,
so we only need to show that every $(I,J)$-submatrix is invertible.

For fixed $I$ and $J$, the submatrix in the previous paragraph is determined by the system
\begin{equation}
\label{transition-equation-two}
\delta_n^r \cdot \vv(A,B) = \sum_{\substack{T \subseteq S \\ |T| = r}} 
r! \cdot \vv(A \cup T, B \cup T)
\end{equation}
where $S := \{1, \dots, n \} - (A \cup B)$.
The system \eqref{transition-equation-two} represents a linear map
\begin{multline}
\mathrm{span} \{ \vv(A,B) \,:\, |A| = i, \, |B| = j, \, A - B = I, \, B - A = J \} \rightarrow   \\
\mathrm{span} \{ \vv(C,D) \,:\, |C| = n-j, \, |D| = n-i, \, C - D = I, \, D - C = J \}
\end{multline}
 where all sets $A, B, C, D$ are subsets of $\{1, \dots, n \}$.
 In particular, the system \eqref{transition-equation-two}
represents a square matrix of size ${n - |I| - |J| \choose k}$ where $k := i - |I| = j - |J|$
is the size of the intersection $|A \cap B|$ for any element $\vv(A,B)$ appearing in the LHS.
If we let $S' := \{1, \dots, n\} - (I \cup J)$, 
The invertibility of the system \eqref{transition-equation-two} is equivalent to the invertibility 
of the system
\begin{equation}
\label{transition-equation-three}
\delta_n^r \cdot \vv(R,R) = \sum_{\substack{T \subseteq S' \\ |T| = r}} 
r! \cdot \vv(R \cup T, R \cup T)
\end{equation}
representing a linear map
\begin{equation}
\mathrm{span} \{ \vv(R,R) \,:\, R \subseteq S' \text{ and } |R| = k \} \rightarrow 
\mathrm{span} \{ \vv(R',R') \,:\, R' \subseteq S' \text{ and } |R'| = k+r \}.
\end{equation}
Observe that 
\begin{equation}
k + (k + r) = (i - |I|) + (j - |J|) + n - i - j = n - |I| - |J| = |S'|
\end{equation}
so that 
the system \eqref{transition-equation-three} is invertible by Theorem~\ref{boolean-incidence}.
\end{proof}

\section{Casimir elements and fermionic diagonal coinvariants}
\label{Casimir}

Let $W$ be an irreducible reflection group with reflection representation $V = \CC^n$.  
Let $\theta_1, \dots, \theta_n$
be a basis of $V$. Given the choice of 
$\theta_1, \dots, \theta_n$, we let  $\xi_1, \dots, \xi_n$ be the {\em dual basis} of $V^*$ characterized by 
\begin{equation}
\xi_i(\theta_j) = \delta_{i,j}.
\end{equation}
We rename the element 
$\delta_n = \theta_1 \xi_1 + \cdots + \theta_n \xi_n$ of 
$\wedge (V \oplus V^*)$ studied in the previous section as $\delta_W$:
\begin{equation}
\delta_W := \delta_n = \theta_1 \xi_1 + \cdots + \theta_n \xi_n \in V \otimes V^* \subseteq 
\wedge (V \oplus V^*).
\end{equation}
We refer to $\delta_W$ as the {\em Casimir element} of $W$.

The full general linear group $\GL(V)$ acts on $\wedge (V \oplus V^*)$ and it is not difficult 
to check using elementary matrices and the dual basis property that $\delta_W$ 
is invariant under this action.  Equivalently, the Casimir element $\delta_W$ is independent of 
the choice of basis $\theta_1, \dots, \theta_n$.
In particular, the element $\delta_W$ lies in the $W$-invariant subring
$\wedge (V \oplus V^*)^{W}$ of $\wedge (V \oplus V^*)$.  In fact,

\begin{proposition}
\label{casimir-generation}
The Casimir element $\delta_W$ generates the $W$-invariant subring
$\wedge (V \oplus V^*)^{W}$ of $\wedge (V \oplus V^*)$.
\end{proposition}

\begin{proof}
Let $G$ be a finite group and let $U, U'$ be irreducible complex representations of $G$.
The tensor product $U \otimes U'$ is a $G$-module by the rule $g.(u \otimes u') := (g.u) \otimes (g.u')$
for $g \in G, u \in U,$ and $u' \in U'$.  (This is the {\em Kronecker product} of the modules $U$ and $U'$.)
Character orthogonality implies 
\begin{align}
\text{multiplicity of the trivial $G$-module in $U \otimes U'$} &= \dim (U \otimes U')^G \\ &=
\begin{cases}
1 & U' \cong U^* \\
0 & \text{otherwise.}
\end{cases}
\end{align}

Since $W$ is an irreducible complex reflection group, a result of Steinberg
(see \cite[Thm. A, \S 24-3, p. 250]{Kane})
implies that the exterior
powers $\wedge^0 V, \wedge^1 V, \dots, \wedge^n V$ are pairwise
nonisomorphic irreducible representations of $W$.
The same is true of their duals $\wedge^0 V^*, \wedge^1 V^*, \dots, \wedge^n V^*$.
Since the $(i,j)$-bidegree of $\wedge (V \oplus V^*)$ is given by 
$\wedge (V \oplus V^*)_{i,j} = \wedge^i V \otimes \wedge^j V^*$, the argument of the last paragraph 
gives
\begin{equation}
\dim \wedge (V \oplus V^*)_{i,j}^W = \begin{cases}
1 & i = j \\
0 & i \neq j
\end{cases}
\end{equation}
for any $0 \leq i, j \leq n$.
On the other hand, we have $\delta_W \in \wedge (V \oplus V^*)^W_{1,1}$ and a quick computation 
shows that the powers $\delta_W^0, \delta_W^1, \dots, \delta_W^n$ are nonzero.
The proposition follows. 
\end{proof}

We are ready to describe the bigraded $W$-module structure of $DR_W$. 
We state our answer in terms of the {\em Grothendieck ring} of $W$. 
Recall that this is the $\ZZ$-algebra generated by the symbols $[U]$ where $U$ is a finite-dimensional
$W$-module and subject to a relation $[U] = [U'] + [U'']$ for any short exact sequence
\begin{equation}
0 \rightarrow U' \rightarrow U \rightarrow U'' \rightarrow 0.
\end{equation}
In particular, if $U \cong U'$ we have $[U] = [U']$.
Multiplication in the Grothendieck ring corresponds to Kronecker product, i.e.
$[U] \cdot [U'] := [U \otimes U']$.

\begin{theorem}
\label{fermion-module-structure}
Let $W$ be an irreducible complex reflection group acting on its reflection representation $V = \CC^n$ and let
$0 \leq i, j \leq n$.
If $i + j > n$ we have $(DR_W)_{i,j} = 0$.  If $i + j \leq n$, inside the Grothendieck group of $W$ we have
\begin{equation}
[ (DR_W)_{i,j} ] =   [ \wedge^i V ] \cdot [ \wedge^j V^*] - [\wedge^{i-1} V ] \cdot [ \wedge^{j-1} V^*]
\end{equation}
where we interpret $\wedge^{-1} V = \wedge^{-1} V^* = 0$.
\end{theorem}

\begin{proof}
Thanks to Proposition~\ref{casimir-generation} we can 
 model $DR_W$ as 
\begin{equation}
\label{dr-decomposition}
DR_W = \wedge (V \oplus V^*) / \langle \delta_W \rangle =
\wedge \{ \Theta_n, \Xi_n \} / \langle \delta_n \rangle.
\end{equation}
If $i = 0$ or $j = 0$, the claim follows since $\delta_W$ lies in bidegree $(1,1)$, so assume $i, j > 0$.

If $i + j \leq n$, let $r = n - i - j + 1$. Theorem~\ref{delta-lefschetz} implies that the multiplication map
\begin{equation}
\delta_W^r \cdot (-) : \wedge (V \oplus  V^*)_{i-1,j-1} \rightarrow 
\wedge (V \oplus V^*)_{n-j+1, n-i+1}
\end{equation}
is a linear isomorphism.  Whenever a composition $f \circ g$ of maps is a bijection, the map $g$ is injective 
and the map $f$ is surjective.  Therefore, the map
\begin{equation}
\delta_W \cdot (-) : \wedge (V \oplus  V^*)_{i-1,j-1} \rightarrow 
\wedge (V \oplus V^*)_{i,j}
\end{equation}
is an injection which we know to be $W$-equivariant.
The claimed decomposition of $[(DR_W)_{i,j}]$ follows from \eqref{dr-decomposition}.

Now suppose $i + j > n$.
By analogous reasoning
\begin{equation}
\delta_W \cdot (-) : \wedge (V \oplus  V^*)_{i-1,j-1} \rightarrow 
\wedge (V \oplus V^*)_{i,j}
\end{equation}
is a $W$-equivariant surjection so that $(DR_W)_{i,j} = 0$.
\end{proof}

\begin{corollary}
\label{dimension-of-dr}
If $W$ has rank $n$,
the vector space dimension of $DR_W$ is 
$\dim DR_W = {2n + 1 \choose n}$.
\end{corollary}

\begin{proof}  Since $\dim \wedge^k V = \dim \wedge^k V^* = {n \choose k}$ where $V$
is the reflection representation of $W$,
Theorem~\ref{fermion-module-structure} yields
\begin{align}
\dim DR_W &= \sum_{k = 0}^n {n \choose k} {n \choose n-k} + \sum_{j = 1}^n {n \choose j-1}{n \choose n - j} \\
&= {2n \choose n} + {2n \choose n-1} \\
&= {2n+1 \choose n}
\end{align}
by the Pascal recursion.
\end{proof}

Recall that the {\em Catalan} and {\em Narayana} numbers are given by
\begin{equation}
\mathrm{Cat}(n) := \frac{1}{n+1}{2n \choose n} \quad \text{and} \quad 
\mathrm{Nar}(n,k) := \frac{1}{n} {n \choose k} {n \choose k-1}.
\end{equation}
We have $\sum_{k = 1}^n \mathrm{Nar}(n,k) = \mathrm{Cat}(n)$.
These numbers have many combinatorial interpretations; for example, $\mathrm{Cat}(n)$ counts
the number of Dyck paths of size $n$ and $\mathrm{Nar}(n,k)$ counts the number of such paths
with $k-1$ peaks.
The Catalan and Narayana numbers show up as the dimensions of the `boundary' pieces
of $DR_W$.

\begin{corollary}
\label{dimension-of-dr-catalan}
If $W$ has rank $n$, we have
\begin{equation}
\dim (DR_W)_{k,n-k} = \mathrm{Nar}(n+1,k+1) 
\end{equation} for $0 \leq k \leq n$ so that 
\begin{equation}
\sum_{k = 0}^n \dim (DR_W)_{k,n-k} = \mathrm{Cat}(n+1).
\end{equation}
\end{corollary}

\begin{proof} Thanks to Theorem~\ref{fermion-module-structure}
one need only verify the identity
\begin{equation}
\mathrm{Nar}(n+1,k+1) = {n \choose k}  {n \choose n-k} - {n \choose k-1} {n \choose n-k-1}, 
\end{equation}
which is straightforward.
\end{proof}

For any reflection group $W$, there are Catalan and Narayana numbers attached to
$W$ (see for example \cite{ARR}); the numbers appearing in Corollary~\ref{dimension-of-dr-catalan}
are their type A instances, and depend only on the rank of $W$.

\begin{remark}
\label{other-g}
The results in this section and the next apply equally well to any finite group $G$ and any 
$n$-dimensional $G$-module $V$ for which the exterior powers
$\wedge^0 V, \wedge^1 V, \dots, \wedge^n V$ are pairwise nonisomorphic irreducibles.
The proofs go through {\em mutatis mutandis}.
\end{remark}

\section{Motzkin paths and standard bases}
\label{Motzkin}

In this section we describe the standard monomial basis of 
$DR_W$ (with respect to a term order $\prec$ which we will define) in terms of a certain family of lattice paths.
A {\em Motzkin path} is a lattice path in $\ZZ^2$ consisting of
up-steps $(1,1)$, down-steps $(1,-1)$, and horizontal steps $(1,0)$ which starts at the origin,
ends on the $x$-axis, and never sinks below the $x$-axis. We consider a variant of Motzkin 
paths which have decorated horizontal steps and need not end on the $x$-axis.

Let $\Pi(n)$ be the family of $n$-step lattice paths $\sigma = (s_1, \dots, s_n)$ in $\ZZ^2$ which start at the origin 
and consist of up-steps $(1,1)$, down-steps $(1,-1)$, and horizontal steps
$(1,0)$ in which each horizontal step is decorated with a $\theta$ or a $\xi$.
We let $\Pi(n)_{\geq 0} \subseteq \Pi(n)$ be the family of paths which never sink below the $x$-axis.
Two paths in $\Pi(9)$ are shown in Figure~\ref{path-figure};
the top path lies in $\Pi(9)_{\geq 0}$ but the bottom path does not.

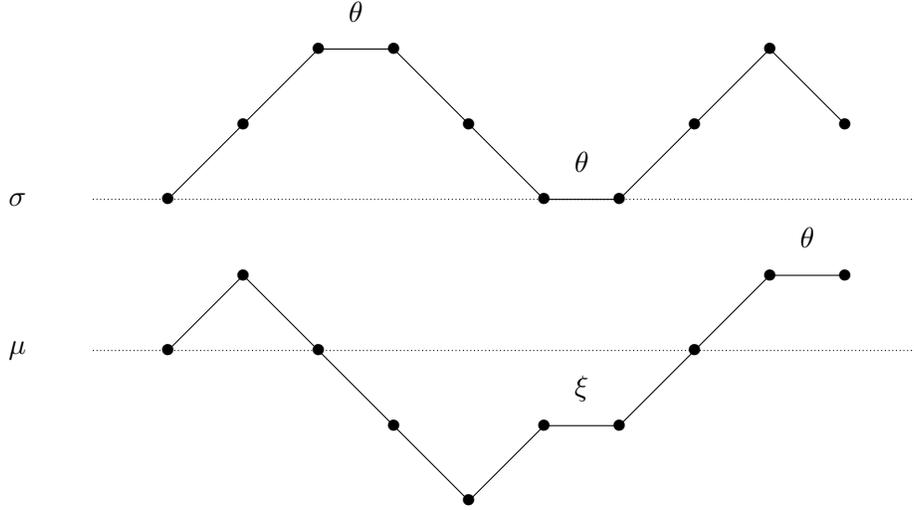
\begin{figure}

\begin{tikzpicture} 
\node (A0) at (-2,0) {$\sigma$};

\node (A1) at (0,0) {$\bullet$};

\node (A2) at (1,1) {$\bullet$};

\node (A3) at (2,2) {$\bullet$};

\node (A4) at (3,2) {$\bullet$};

\node (A5) at (4,1) {$\bullet$};

\node (A6) at (5,0) {$\bullet$};

\node (A7) at (6,0) {$\bullet$};

\node (A8) at (7,1) {$\bullet$};

\node (A9) at (8,2) {$\bullet$};

\node (A10) at (9,1) {$\bullet$};

\node at (2.5,2.5) {$\theta$};

\node at (5.5,0.5) {$\theta$};

\draw (0,0) -- (1,1) -- (2,2) -- (3,2) -- (4,1) -- (5,0) -- (6,0) -- (7,1) -- (8,2) -- (9,1);

\draw [densely dotted] (-1,0) -- (10,0);

\end{tikzpicture}

\begin{tikzpicture}
\node (A0) at (-2,0) {$\mu$};

\node (A1) at (0,0) {$\bullet$};

\node (A2) at (1,1) {$\bullet$};

\node (A3) at (2,0) {$\bullet$};

\node (A4) at (3,-1) {$\bullet$};

\node (A5) at (4,-2) {$\bullet$};

\node (A6) at (5,-1) {$\bullet$};

\node (A7) at (6,-1) {$\bullet$};

\node (A8) at (7,0) {$\bullet$};

\node (A9) at (8,1) {$\bullet$};

\node (A10) at (9,1) {$\bullet$};

\node at (5.5,-0.5) {$\xi$};

\node at (8.5,1.5) {$\theta$};

\draw (0,0) -- (1,1) -- (2,0) -- (3,-1) -- (4,-2) -- (5,-1) -- (6,-1) -- (7,0) -- (8,1) -- (9,1);

\draw [densely dotted] (-1,0) -- (10,0);

\end{tikzpicture}

\caption{Two paths in $\Pi(9)$.}
\label{path-figure}
\end{figure}

The {\em depth} $d(\sigma)$ of a path $\sigma \in \Pi(n)$ is the minimum $y$-value attained by $\sigma$.
If $\sigma$ and $\mu$ are as in Figure~\ref{path-figure} then
$d(\sigma) = 0$ and $d(\mu) = -2$.
We have
\begin{equation}
\Pi(n)_{\geq 0} = \{ \sigma \in \Pi(n) \,:\, d(\sigma) = 0 \}
\end{equation}
and $d(\sigma) < 0$ for any $\sigma \in \Pi(n) - \Pi(n)_{\geq 0}$.

Let $\sigma = (s_1, \dots, s_n) \in \Pi(n)$.  The {\em weight} of the $i^{th}$ step $s_i$ of $\sigma$ is 
\begin{equation}
\wt(\sigma) := \begin{cases}
1 & \text{if $s_i = (1,1)$ is an up-step} \\
\theta_i & \text{if $s_i = (1,0)$ is decorated with $\theta$} \\
\xi_i & \text{if $s_i = (1,0)$ is decorated with $\xi$} \\
\theta_i \xi_i & \text{if $s_i = (1,-1)$ is a down-step}
\end{cases}
\end{equation}
and the weight of $\sigma$ is the product
\begin{equation}
\wt(\sigma) := \wt(s_1) \cdots \wt(s_n)
\end{equation}
of the steps of $\sigma$ in the order in which they appear.  For the paths $\sigma$ and $\mu$
in Figure~\ref{path-figure} we have
\begin{equation}
\wt(\sigma) = \theta_3 \cdot \theta_4 \xi_4 \cdot \theta_5 \xi_5 \cdot \theta_6 \cdot \theta_9 \xi_9
\quad \text{and} \quad
\wt(\mu) = \theta_2 \xi_2 \cdot \theta_3 \xi_3 \cdot \theta_4 \xi_4 \cdot \xi_6 \cdot \theta_9.
\end{equation}

A moment's thought shows that 
$\sigma \mapsto \wt(\sigma)$ gives a bijection from $\Pi(n)$ to the set of monomials in 
$\wedge \{ \Theta_n, \Xi_n\}$, where monomials with differing signs are considered equivalent.
We will identify paths $\sigma$ with their monomials $\wt(\sigma)$.

The {\em (total) degree} of a path $\sigma$ is
\begin{equation}
\deg(\sigma) := n - \text{(the terminal $y$-coordinate of $\sigma$)}.
\end{equation}
This is simply the total number of exterior generators $\theta_i$ and $\xi_i$ appearing in the monomial $\sigma$.
We define the {\em $\theta$-degree} $\deg_{\theta}(\sigma)$ and 
{\em $\xi$-degree} $\deg_{\xi}(\sigma)$ analogously.  Combinatorially,
\begin{equation}
\deg_{\theta}(\sigma) = \text{(number of down-steps)}  + \text{(number of $\theta$-horizontal steps)}
\end{equation}
and 
\begin{equation}
\deg_{\xi}(\sigma) = \text{(number of down-steps)}  + \text{(number of $\xi$-horizontal steps)}.
\end{equation}
If $\sigma$ and $\mu$ are as in Figure~\ref{path-figure} then
\begin{equation}
\begin{cases}
\deg(\sigma) = 8 \\
\deg_{\theta}(\sigma) = 5 \\
\deg_{\xi}(\sigma) = 3
\end{cases}  \quad \text{and} \quad
\begin{cases}
\deg(\mu) = 8 \\
\deg_{\theta}(\mu) = 4 \\
\deg_{\xi}(\mu) = 4
\end{cases}
\end{equation}

We introduce the  total order $\prec$ on paths $\sigma \in \Pi(n)$, or on monomials in 
$\wedge \{ \Theta_n, \Xi_n \}$ given by
\begin{equation}
\sigma \prec \sigma' \Leftrightarrow
\begin{cases}
\deg(\sigma) < \deg(\sigma') & \text{or} \\
\deg(\sigma) = \deg(\sigma') \text{ and } d(\sigma) > d(\sigma') & \text{or} \\
\deg(\sigma) = \deg(\sigma') \text{ and } d(\sigma) = d(\sigma') \text{ and }
\sigma <_{\mathrm{lex}} \sigma'
\end{cases}
\end{equation}
where in the last branch $<_\mathrm{lex}$ means the lexicographical order on
the paths $\sigma = (s_1, \dots, s_n)$
and $\sigma' = (s_1', \dots, s_n')$ induced by declaring the step order
\begin{equation}
\label{step-lex}
(1,1) < \text{$(1,0)$ with $\theta$-decoration} < \text{$(1,0)$ with $\xi$-decoration} < (1,-1).
\end{equation}
The collection of paths/monomials with a given bidegree $(i,j)$ form a subinterval of 
$\prec$ for all $0 \leq i, j \leq n$.
In our running example of Figure~\ref{path-figure}, we have $\deg(\sigma) = \deg(\mu)$
but $d(\sigma) > d(\mu)$ so that $\sigma \prec \mu$.

\begin{lemma}
\label{is-term-order}
The total order $\prec$ is a term order for $\wedge \{ \Theta_n, \Xi_n \}$.
\end{lemma}

\begin{proof}
The first branch of the definition of $\prec$ guarantees that the monomial $1$ with path
consisting of a sequence of $n$ up-steps is the minimum monomial under $\prec$.
Checking that $\prec$ respects multiplication amounts to the observation that total degree,
depth, and lexicographical order are all respected by multiplication.
\end{proof}

It turns out that 
the set $\{ \wt(\sigma) \,:\, \sigma \in \Pi(n)_{\geq 0} \}$ descends to a $\CC$-basis of 
$DR_W$.  In fact, we prove something stronger.

\begin{theorem}
\label{standard-monomial-basis}
The set $\{ \wt(\sigma) \,:\, \sigma \in \Pi(n)_{\geq 0} \}$ is the standard monomial basis of $DR_W$
with respect to $\prec$.
\end{theorem}

\begin{proof}
Let $I_n = \langle \delta_n \rangle \subseteq \wedge \{ \Theta_n, \Xi_n \}$ be the defining ideal of $DR_W$
(here we apply Proposition~\ref{casimir-generation}).  Identifying paths with monomials, we want to show
$N(I_n) = \Pi(n)_{\geq 0}$ with respect to $\prec$.  We proceed by induction on $n$, with the base case $n = 1$
being immediate.

Suppose $n > 1$ and $\sigma = (s_1, \dots, s_n) \in \Pi(n) - \Pi(n)_{\geq 0}$. In particular, we have $d(\sigma) < 0$.
The following lemma will show inductively that $\sigma \notin N(I_n)$.

\begin{lemma}
\label{sigma-reduction}
The monomial $\sigma$ lies in $\LM(I_n)$ or else $\sigma = 0$ in the quotient $DR_W$.
\end{lemma}
\begin{proof} (of Lemma~\ref{sigma-reduction})
Let $\sigma_0 \in \wedge \{ \Theta_{n-1}, \Xi_{n-1} \}$ be the monomial $\sigma$ with its last step $s_n$
removed. The proof breaks into cases depending on the step $s_n$.

{\bf Case 1:} {\em The last step $s_n$ is a horizontal step (of either decoration $\theta$ or $\xi$).}

We assume the decoration of $s_n$ is $\theta$; the other case is similar.
In this case, we have $\sigma_0 \in \Pi(n-1) - \Pi(n-1)_{\geq 0}$ and $\sigma = \sigma_0 \theta_n$.  
We may inductively assume that 
$\sigma_0 \in \LM(I_{n-1})$ so that $\sigma_0 = \LM(f  \cdot \delta_{n-1})$ for some polynomial 
$f \in \wedge \{ \Theta_{n-1}, \Xi_{n-1} \}$.  Since
\begin{equation}
f \cdot \delta_n \cdot \theta_n =  f \cdot \delta_{n-1} \theta_n + f \cdot \theta_n \xi_n \cdot \theta_n 
= f \cdot \delta_{n-1} \cdot \theta_n,
\end{equation}
we conclude that $f \cdot \delta_{n-1} \cdot \theta_n \in I_n$. We have
\begin{equation}
\LM(f \cdot \delta_{n-1} \cdot \theta_n) = \LM(f \cdot \delta_{n-1}) \cdot \theta_n = \sigma_0 \cdot \theta_n = \sigma,
\end{equation}
completing the proof of Case 1.

{\bf Case 2:} {\em The last step $s_n$ is a down-step $(1,-1)$.}

If $\sigma_0 \in \Pi(n-1) - \Pi(n-1)_{\geq 0}$ sinks below the $x$-axis, the proof is similar to that of 
Case 1. One right-multiplies $f \cdot \delta_n$ by $\theta_n \xi_n$ instead of $\theta_n$; we leave the details to the 
reader.

In this case we could have $\sigma_0 \in \Pi(n-1)_{\geq 0}$, but this would imply that $\sigma_0$ ends on the 
$x$-axis, so that $\deg(\sigma) = \deg(\sigma_0) + 2 = (n-1) + 2 = n+1$.  
Theorem~\ref{fermion-module-structure} then  forces $\sigma = 0$ in the quotient $DR_W$, completing the 
proof of Case 2.

{\bf Case 3:} {\em The last step $s_n$ is an up-step $(1,1)$.}

This is the most involved case. We have $\sigma_0 = \sigma$ and $\sigma_0 \in \Pi(n-1) - \Pi(n-1)_{\geq 0}$.
By induction, we may assume that there is $f \in \wedge \{ \Theta_{n-1}, \Xi_{n-1} \}$ with 
$\sigma_0 = \LM(f \cdot \delta_{n-1})$.  Now consider 
\begin{equation}
\label{test-polynomial}
f \cdot \delta_n = f \cdot \delta_{n-1} + f \cdot \theta_n \xi_n \in I_n.
\end{equation}
By discarding redundant terms if necessary, we may assume that $f$ is bi-homogeneous.
The monomial $\sigma = \sigma_0$ is the $\prec$-largest monomial appearing in $f \cdot \delta_{n-1}$.
Since $\sigma$ does not involve $\theta_n$ or $\xi_n$, it does not appear in $f \cdot \theta_n \xi_n$.
We will have $\sigma = \LM(f \cdot \delta_n)$ unless some monomial $\mu$ appearing in $f \cdot \theta_n \xi_n$
satisfies $\mu \succ \sigma$.  

Let $\mu$ be the $\prec$-largest element of $f \cdot \theta_n \xi_n$ and assume $\sigma \prec \mu$.
Let $\mu_0 \in \Pi(n-1)$ be the path obtained from $\mu$ by removing its last step (which is necessarily
a down-step since $\mu$ appears in $f \cdot \theta_n \xi_n$).  Since $\sigma \prec \mu$, the bihomogeneity
of $f$ forces 
$d(\mu) \leq d(\sigma) < 0$.

{\bf Subcase 3.1:}  {\em We have $\mu_0 \in \Pi(n-1)_{\geq 0}$, or equivalently $d(\mu_0) \geq 0$.}

Since $d(\mu) = d(\mu_0) + 1 < 0$, this can only happen if $d(\mu_0) = 0$ and $\mu_0$ ends at the lattice 
points $(n-1,0)$.  This implies that $\deg(\mu) = \deg(\mu_0) + 2 = (n-1) + 2 = n+1$ and
Theorem~\ref{fermion-module-structure} forces $\mu \in I_n$.  
We may therefore discard the term involving $\mu$ from
 \eqref{test-polynomial} and still have an element of $I_n$ involving $\sigma$.
 
 {\bf Subcase 3.2:}  {\em We have $\mu_0 \in \Pi(n-1) - \Pi(n-1)_{\geq 0}$, or equivalently $d(\mu_0) < 0$.}
 
 In this case, we induct on $n$ to obtain some polynomial $g \in \wedge \{ \Theta_{n-1}, \Xi_{n-1} \}$ 
 whose leading monomial is $\mu_0 = \LM( g \cdot \delta_{n-1})$.  We calculate
 \begin{equation}
 \LM (g \cdot \delta_n \cdot \theta_n \xi_n) = \LM(g \cdot \delta_{n-1} \cdot \theta_n \xi_n) =
 \LM(g \cdot \delta_{n-1}) \cdot \theta_n \xi_n = \mu_0 \cdot \theta_n \xi_n = \mu
 \end{equation}
 where the second equality used the fact that $g \cdot \delta_{n-1}$ does not involve $\theta_n$ or $\xi_n$.
Since $\sigma$ does not involve $\theta_n$ or $\xi_n$, it does not appear in 
$g \cdot \delta_n \cdot \theta_n \xi_n$.  We may therefore replace \eqref{test-polynomial} by 
\begin{equation}
\label{test-polynomial-two}
f \cdot \delta_{n-1} + (f - g \cdot \delta_{n-1}) \cdot \theta_n \xi_n \in I_n
\end{equation}
to obtain another element of $I_n$ which involves $\sigma$ only in its first term, still satisfies
$\sigma = \LM(f \cdot \delta_{n-1})$, but now only involves monomials $\prec \mu$.

Iterating the arguments of Subcases 3.1 and 3.2, we see that $\sigma \in \LM(I_n)$, proving both
Case 3 and the lemma.
\end{proof}
We complete the proof of Theorem~\ref{standard-monomial-basis} using Lemma~\ref{sigma-reduction}.
Lemma~\ref{sigma-reduction} implies $N(I_n) \subseteq \Pi(n)_{\geq 0}$, and to force equality it suffices to verify
\begin{equation}
\dim DR_W = | \Pi(n)_{\geq 0} |. 
\end{equation} 
In fact, we verify the equality of bigraded Hilbert series
\begin{equation}
\Hilb(DR_W; q, t) = \sum_{\sigma \in \Pi(n)_{\geq 0}} q^{\deg_{\theta}(\sigma)} t^{\deg_{\xi}(\sigma)} =: P_n(q,t).
\end{equation}

If we let $\Pi(n)_{= 0} \subseteq \Pi(n)_{\geq 0}$ be the subset of paths that end on the $x$-axis and let
\begin{equation}
P'_n(q,t) := \sum_{\sigma \in \Pi(n)_{= 0}}  q^{\deg_{\theta}(\sigma)} t^{\deg_{\xi}(\sigma)},
\end{equation}
considering the addition of one more step to a path yields
\begin{equation}
\label{path-recursion}
P_{n+1}(q,t) = (1 + q + t + qt) \cdot P_n(q,t) - (qt) \cdot P_n'(q,t).
\end{equation}
On the other hand (adopting the notation $DR_{W(n)}$ for $DR_W$ whenever $W$ has rank $n$) 
Theorem~\ref{fermion-module-structure} yields
\begin{equation}
\label{piece-dimension}
\dim (DR_{W(n+1)})_{i,j} = \begin{cases}
{n + 1 \choose i} \cdot {n+1 \choose j} - {n + 1 \choose i-1} \cdot {n+1 \choose j-1} &
 \text{if $i, j > 0$ and $i + j \leq n+1$} \\
{n + 1 \choose i} \cdot {n+1 \choose j} & \text{if $i = 0$ or $j = 0$}  \\
0 & \text{if $i + j > n+1$}
\end{cases}
\end{equation}
It can be shown using the Pascal identity and Equation~\eqref{piece-dimension} that 
\begin{equation}
\Hilb(DR_{W(n+1)};q,t) = (1 + q + t + qt) \cdot \Hilb(DR_{W(n)};q,t)  - 
(qt) \cdot \sum_{i+j = n +1} \dim (DR_{W(n)})_{i,j} \cdot q^i t^j,
\end{equation}
which matches the combinatorial recursion in Equation~\eqref{path-recursion}.
\end{proof}

In the course of proving Theorem~\ref{standard-monomial-basis}, we derived the following 
combinatorial expression for the bigraded Hilbert series of $DR_W$.

\begin{corollary}
\label{hilbert-motzkin}
If $W$ has rank $n$, we have 
\begin{equation}
\Hilb(DR_W; q,t) = \sum_{\sigma \in \Pi(n)_{\geq 0}} q^{\deg_{\theta}(\sigma)} t^{\deg_{\xi}(\sigma)}.
\end{equation}
\end{corollary}

\section{The permutation representation of $\symm_n$}
\label{Permutation}

In the coinvariant theory of the symmetric group $\symm_n$, it is more common to consider 
its $n$-dimensional permutation representation $U$ as opposed to 
its $(n-1)$-dimensional reflection representation $V$
In this  section we describe how to translate our results into this setting.

The following decompositions of $U$ and $U^*$ into $\symm_n$-irreducibles are well-known:
\begin{equation}
U = V \oplus U^{\symm_n} \quad \text{and} \quad
U^* = V^* \oplus (U^*)^{\symm_n}.
\end{equation}
It follows that 
\begin{align}
\wedge (U \oplus U^*) &\cong \wedge [ ( V \oplus U^{\symm_n}) \oplus (V^* \oplus (U^*)^{\symm_n})]  \\
&\cong \wedge[ (V \oplus V^*) \oplus (U^{\symm_n} \oplus (U^*)^{\symm_n})]  \\
&\cong [\wedge(V \oplus V^*)] \otimes [\wedge( U^{\symm_n} \oplus (U^*)^{\symm_n})].
\end{align}
Modding out by ideals generated by $\symm_n$-invariants with vanishing constant term, we see that 
\begin{equation}
\label{final-isomorphism}
\wedge(U \otimes U^*)/
\langle \wedge (U \otimes U^*)^{\symm_n}_+ \rangle \cong
\wedge(V \otimes V^*)/
\langle \wedge (V \otimes V^*)^{\symm_n}_+ \rangle.
\end{equation}

Let $\symm_n$ act on $\wedge\{ \Theta_n, \Xi_n \}$ diagonally, viz.
$w.\theta_i := \theta_{w(i)}$ and $w.\xi_i := \xi_{w(i)}$.
Expressing the left-hand side of \eqref{final-isomorphism} in terms of coordinates, we have the following 
translation of Theorem~\ref{fermion-module-structure}, 
Corollary~\ref{dimension-of-dr}, and Corollary~\ref{dimension-of-dr}.

\begin{theorem}
\label{permutation-structure-theorem}
Let $DR_n$ be the bigraded $\symm_n$-module
\begin{equation}
DR_n := \wedge\{ \Theta_n, \Xi_n \}/\langle \wedge\{ \Theta_n, \Xi_n \}^{\symm_n}_+ \rangle.
\end{equation}
We have $(DR_n)_{i,j} = 0$ whenever $i + j \geq n$.  If $i + j < n$, we have
\begin{equation}
\Frob (DR_n)_{i,j} = s_{(n-i,1^i)} * s_{(n-j,1^j)} - s_{(n-i+1,1^{i-1})} * s_{(n-j+1,1^{j-1})}
\end{equation}
where $*$ denotes Kronecker product.
Here we interpret $s_{(n+1,-1)} = 0$.  We have
\begin{equation}
\label{zabrocki-conjecture}
\dim DR_n = {2n -1 \choose n}
\end{equation}
and, for $1 \leq k \leq n$, we have
\begin{equation}
\dim (DR_n)_{k-1,n-k} = \mathrm{Nar}(n,k)
\end{equation}
so that $\sum_{k = 1}^n \dim (DR_n)_{k-1,n-k} = \mathrm{Cat}(n)$.
\end{theorem}

Equation~\eqref{zabrocki-conjecture} was conjectured by Mike Zabrocki \cite{ZabrockiFermion}
for the Open Problems in Algebraic Combinatorics 2020 Conference\footnote{which was delayed due to COVID-19, 
but will hopefully take place soon}.
We also have a lattice path basis of the $\symm_n$-module $DR_n$  
in Theorem~\ref{permutation-structure-theorem}.
For a partition $\lambda \vdash n$, work of Rosas \cite{Rosas} implies that 
\begin{equation}
\langle \grFrob(DR_n; q, t), s_{\lambda} \rangle = 0
\end{equation}
unless the partition $\lambda = (\lambda_1 \geq \lambda_2 \geq \lambda_3  \geq \cdots )$ satisfies $\lambda_3 \leq 2$
(i.e. the Young diagram of $\lambda$ is a union of two possibly empty hooks).
While these multiplicities can be less than aesthetic in general, they are nice when 
$\lambda$ is a hook.
Recall that the {\em $q,t$-analog} of $n$ is given by
\begin{equation}
[n]_{q,t} := \frac{q^n-t^n}{q-t} = q^{n-1} + q^{n-2} t + \cdots + q t^{n-2} + t^{n-1}.
\end{equation}

\begin{proposition}
\label{hook-multiplicities}
The graded multiplicities of the trivial and sign representations in $DR_n$ are given by
\begin{equation}
\langle \grFrob(DR_n;q,t), s_{(n)} \rangle = 1 \quad \text{and} \quad
\langle \grFrob(DR_n;q,t), s_{(1^n)} \rangle = [n]_{q,t}.
\end{equation}
If $0 < k < n-1$ we have
\begin{equation}
\label{dr-hook-multiplicities}
\langle \grFrob(DR_n;q,t), s_{(n-k,1^k)} \rangle = [k+1]_{q,t} + (qt) \cdot [k]_{q,t}.
\end{equation}
\end{proposition}

\begin{proof}
The equation $\langle \grFrob(DR_n;q,t), s_{(n)} \rangle = 1$ is immediate since $DR_n$ 
is obtained from $\wedge \{ \Theta_n, \Xi_n \}$ by modding out by $\symm_n$-invariants with vanishing constant term.
The multiplicity of the signed representation follows from Theorem~\ref{permutation-structure-theorem}
and the fact that for any partitions $\lambda, \mu \vdash n$
\begin{equation}
\text{multiplicity of $s_{(1^n)}$ in $s_{\lambda} * s_{\mu}$} = \begin{cases}
1 & \text{if $\mu = \lambda'$} \\
0 & \text{otherwise}
\end{cases}
\end{equation}
where $\lambda'$ is the conjugate (transpose) partition of $\lambda$.

We turn our attention to Equation~\eqref{dr-hook-multiplicities}.
For any statement $P$, let $\chi(P) = 1$ if $P$ is true and $\chi(P) = 0$ if $P$ is false.
Rosas proves \cite[Proof of Thm. 13 (4)]{Rosas} that the multiplicity of the Schur function 
$s_{(n-c,1^c)}$ in the Kronecker product $s_{(n-a,1^a)} * s_{(n-b,1^b)}$ is
\begin{equation}
\label{rosas-formula}
\langle s_{(n-a,1^a)} * s_{(n-b,1^b)}, s_{(n-c,1^c)} \rangle = 
\chi(|b-a| \leq c) \times \chi(c \leq a+b \leq 2n - c - 2)
\end{equation}
whenever $0 < a, b < n$ and $0 < c < n-1$.

 For any $0 \leq k \leq n-1$ and all $i + j < n$, we have
\begin{multline}
\langle \Frob(DR_n)_{i,j}, s_{(n-k,1^k)} \rangle =  \\
\langle s_{(n-i,1^i)} * s_{(n-j,1^j)}, s_{(n-k, 1^k)} \rangle - 
\langle s_{(n-i+1,1^{i-1})} * s_{(n-j+1,1^{j-1})}, s_{(n-k, 1^k)} \rangle
\end{multline}
A somewhat tedious casework using Equation~\eqref{rosas-formula} yields
\begin{equation}
\langle \Frob (DR_n)_{i,j},  s_{(n-k,1^k)} \rangle = 
\begin{cases}
1 & \text{if $i+j = k$} \\
1 & \text{if $i+j = k+1$ and $i,j > 0$} \\
0 & \text{otherwise}
\end{cases}
\end{equation}
which is equivalent to Equation~\eqref{dr-hook-multiplicities}.
\end{proof}

In order to state a $DR_n$-analog of Theorem~\ref{standard-monomial-basis},
we need some notation.
We define the {\em primed weight} $\wt'(s)$ of a step $s$ to be
\begin{equation}
\begin{cases}
1 & \text{if $s = (1,1)$ is an up-step} \\
\theta_i & \text{if $s = (1,0)$ is decorated with $\theta$} \\
\xi_i' & \text{if $s = (1,0)$ is decorated with $\xi$} \\
\theta_i \xi_i' & \text{if $s = (1,-1)$ is a down-step}
\end{cases}
\end{equation}
where 
\begin{equation}
\xi_i' := \xi_i + \sum_{j = 2}^n \xi_j.
\end{equation}
The {\em primed weight} $\wt'(\sigma)$ of a path $\sigma \in \Pi(n)$ with steps
$\sigma = (s_1, \dots, s_n)$ is $\wt'(\sigma) := \wt'(s_1) \cdots \wt'(s_n)$.  
Let $\Pi(n)_{> 0} \subseteq \Pi(n)$ consist of those paths which only meet the $x$-axis
at their starting point $(0,0)$ and stay strictly above the $x$-axis otherwise.

\begin{theorem}
\label{permutation-basis}
The set $\{ \wt'(\sigma) \,:\, \sigma \in \Pi(n)_{> 0} \}$ descends to a basis of $DR_n$.
Consequently, we have
\begin{equation}
\Hilb(DR_n;q,t)  = \sum_{\sigma \in \Pi(n)_{> 0}} q^{\deg_{\theta}(\sigma)} t^{\deg_{\xi}(\sigma)}.
\end{equation}
\end{theorem}

\begin{proof}
Proposition~\ref{casimir-generation} and the discussion prior to Theorem~\ref{permutation-structure-theorem}
imply that 
the invariant subalgebra $\wedge \{ \Theta_n, \Xi_n \}^{\symm_n}$ is generated by the three elements
\begin{equation*}
\theta_1 + \cdots + \theta_n, \, \, \xi_1 + \cdots + \xi_n, \, \, \text{ and } \, \, 
\theta_1 \xi_1 + \cdots + \theta_n \xi_n
\end{equation*}
and consequently
\begin{equation}
DR_n = \wedge \{ \theta_1, \dots, \theta_n, \xi_1, \dots, \xi_n \} / 
\langle \theta_1 + \cdots + \theta_n, \xi_1 + \cdots + \xi_n, \theta_1 \xi_1 + \cdots + \theta_n \xi_n \rangle.
\end{equation}

We express $DR_n$ as a successive quotient
\begin{align}
	DR_n &= \wedge \{ \theta_1, \dots, \theta_n, \xi_1, \dots, \xi_n \} / \langle  \theta_1 + \cdots + \theta_n ,  \xi_1 + \cdots +  \xi_n, \theta_1 \xi_1 + \cdots + \theta_n \xi_n \rangle.\\
	&= \left(\wedge \{ \theta_1, \dots, \theta_n\}/\langle \sum_{i=1}^{n} \theta_i  \rangle  \otimes \wedge \{ \xi_1, \dots, \xi_n\}/\langle  \sum_{i=1}^{n} \xi_i  \rangle \right)  / \langle \sum_{i=1}^{n} \theta_i \otimes \xi_i \rangle
\end{align}
Then as graded vector spaces, we identify $\theta_1 = -\theta_2 - \cdots -\theta_n$ and $\xi_1 = -\xi_2 - \cdots - \xi_n$ to obtain
\begin{align}
DR_n &\cong  \left(\wedge \{ \theta_2, \dots, \theta_n\}  \otimes \wedge \{ \xi_2, \dots, \xi_n\} \right)  / \langle (-\theta_2 - \cdots - \theta_n) \otimes (-\xi_2 - \cdots - \xi_n) + \sum_{i=2}^{n} \theta_i \otimes \xi_i \rangle\\
&= \left(\wedge \{ \theta_2, \dots, \theta_n\}  \otimes \wedge \{ \xi_2, \dots, \xi_n\} \right)  / \langle \sum_{i=2}^{n} \theta_i \otimes (\xi_i + \sum_{j=2}^{n} \xi_j  ) \rangle
\end{align}
The transition matrix from the set 
$\{\xi_2 + \sum_{j=2}^{n} \xi_j, \dots, \xi_n + \sum_{j=2}^{n} \xi_j\} = \{ \xi'_2, \dots, \xi'_n\}$ 
to the standard basis $\{ \xi_2, \dots, \xi_n \}$ of the degree 1 component of  $\wedge \{ \xi_2, \dots, \xi_n\}$ 
is
\begin{align*}
\begin{pmatrix}
2 & 1 & \cdots & 1\\
1 & 2 & \cdots & 1\\
\vdots & \vdots & \ddots & \vdots\\
1 & 1 & \cdots & 2
\end{pmatrix}
\end{align*}
which is easily checked to be invertible.  Therefore, the set $\{\xi_2', \dots, \xi_n'\}$ is also a basis
of the degree $1$ component of $\wedge \{ \xi_2, \dots, \xi_n \}$ and we may write
\begin{equation}
DR_n \cong \wedge \{ \theta_2, \dots, \theta_n, \xi_2', \dots, \xi_n' \} / 
\langle \theta_2 \xi_2' + \cdots + \theta_n \xi_n' \rangle.
\end{equation}
Theorem~\ref{standard-monomial-basis} applies to complete the proof.
\end{proof}

\section{Open Problems}
\label{Open}

The key result underpinning our analysis of $DR_W$ and $DR_n$ was the Lefschetz
Theorem~\ref{delta-lefschetz}.
Our proof was combinatorial and ultimately relied on the Boolean poset $B(n)$.
Given the importance of Lefschetz elements in geometry, it is natural to ask the following.

\begin{question}
Is there a geometric proof of Theorem~\ref{delta-lefschetz}?
\end{question}

Modern variants of the HLP and PD were used to great effect in the work of Adiprasito,
Huh, and Katz on the Chow rings of matroids \cite{AHK}.
Is there a deeper meaning to the HLP and PD as they apply to exterior algebras?
Perhaps the realization of $\wedge \{ \Theta_n, \Xi_n \}$ as the holomorphic tangent space to the
origin in $\CC^n \oplus \CC^n$ would be relevant here.

It may also be interesting to consider combining two sets of commuting and anticommuting variables
to get a ring
\begin{equation}
\CC[X_n, Y_n] \otimes \wedge \{ \Theta_n, \Xi_n\} :=
\CC[x_1, \dots, x_n, y_1, \dots, y_n] \otimes \wedge \{ \theta_1, \dots, \theta_n, \xi_1, \dots, \xi_n \}
\end{equation}
which may be identified with the algebra of polynomial-valued holomorphic differential forms on $\CC^n \oplus \CC^n$.
This ring is quadruply graded, and the diagonal action of $\symm_n$ gives rise to a coinvariant space
$\CC[X_n, Y_n, \Theta_n, \Xi_n] / \langle \CC[X_n, Y_n, \Theta_n, \Xi_n]^{\symm_n}_+ \rangle$.
Setting the $\xi$-variables to zero, Zabrocki \cite{Zabrocki} conjectured that the triply graded Frobenius
image of this quotient is given by the {\em Delta Conjecture} of Haglund, Remmel, and Wilson \cite{HRW}.
Furthermore, again when the $\xi$-variables are set to zero,
 Haglund and Sergel \cite{HS} have a conjectural monomial basis of this 
quotient which would extend a basis of the diagonal coinvariants due to Carlsson and Oblomkov \cite{CO}.

\begin{problem}
\label{basis-problem}
Find a  basis of the quotient 
$\CC[X_n, Y_n, \Theta_n, \Xi_n] / \langle \CC[X_n, Y_n, \Theta_n, \Xi_n]^{\symm_n}_+ \rangle$
which generalizes the basis of 
$\CC[X_n, Y_n] / \langle \CC[X_n, Y_n]^{\symm_n}_+ \rangle$ due to Carlsson-Oblomkov \cite{CO}
and the conjectural basis of 
$\CC[X_n, Y_n, \Theta_n] / \langle \CC[X_n, Y_n, \Theta_n]^{\symm_n}_+ \rangle$
due to Haglund-Sergel \cite{HS}.
\end{problem}

A solution to Problem~\ref{basis-problem} might be obtained by interpolating between the 
parking function `schedules' present in \cite{CO, HS} and our Motzkin-like paths $\Pi(n)_{> 0}$.

Let $X_{k \times n} = (x_{i,j})_{1 \leq i \leq k, 1 \leq j \leq n}$ be a $k \times n$ matrix of commuting variables 
and let $\CC[X_{k \times n}]$ be the polynomial ring in these variables.
The ring $\CC[X_{k \times n}]$ carries a $\symm_n$-module structure inherited from column permutation
and the quotient
$\CC[X_{k \times n}] / \langle \CC[X_{k \times n}]^{\symm_n}_+ \rangle$
is a $(\ZZ_{\geq 0})^k$-graded $\symm_n$-module.
When $k = 2$, we recover the classical diagonal coinvariants. F. Bergeron  has many fascinating conjectures
about this object obtained by letting the parameter $k$ grow \cite{Bergeron}.

We can carry out the contstruction of the previous paragraph with a  matrix
$\Theta_{k \times n} = (\theta_{i,j})_{1 \leq i \leq k, 1 \leq j \leq n}$ of anticommuting variables.
We still have an action of $\symm_n$ on columns and can still consider the quotient
\begin{equation}
R(k \times n) := \wedge \{ \Theta_{k \times n} \} / \langle \wedge \{ \Theta_{k \times n} \}^{\symm_n}_+ \rangle.
\end{equation}
In the case $k = 2$ we recover $DR_n$.
For stability results involving such quotients, and corresponding quotients using 
both commuting and anticommuting variables, see \cite{PRR}.

\begin{question}
\label{multigraded}
Find the multigraded isomorphism type of $R(k \times n)$.
\end{question}

It is unclear how to use Lefschetz Theory to solve Question~\ref{multigraded} for $k > 2$.  For any set 
$S \subseteq  \{1, 2, \dots, k \}$ of rows, we have a $\symm_n$-invariant
\begin{equation}
\delta_S := \prod_{i \in S} \theta_{i,1} + \prod_{i \in S} \theta_{i,2} + \cdots + \prod_{i \in S} \theta_{i,n} \in
\wedge \{ \Theta_{k \times n} \}
\end{equation}
where the products are taken in increasing order of $i \in S$.  When $|S|$ is even, this has the potential
to be a Lefschetz element, but $\delta_S^2 = 0$ when $|S|$ is odd.
For $|S| = 1$, the row sum $\delta_S$ may be easy to handle, but the situation becomes more
complicated as an odd-sized set $S$ grows.
Furthermore, one would have to understand how the various images of multiplication
by the $\delta_S$ between bidegrees intersect as $S$ varies.

\section{Acknowledgements}

B. Rhoades was partially supported by NSF Grant DMS-1500838. 
The authors thank Jim Haglund, Eugene Gorsky, Vic Reiner, Vasu Tewari, and Mike Zabrocki for helpful 
conversations and help with references.
B. Rhoades is very grateful to
Vic Reiner for a very helpful Zoom conversation about, among other things, Casimir elements.

\end{document}